\documentclass[apsn,twocolumn,floatfix,groupedaddress]{revtex4} % twocolumn,
\usepackage{graphicx, bm}
\usepackage{amsmath}
\usepackage{amssymb}
\usepackage{mathrsfs}
\usepackage{color}
\usepackage{hyperref}
\usepackage{longtable}
\usepackage{sasha_prih}
\usepackage{amsthm}  %    prl 

\newtheorem{thm}{Theorem} %[section]
\newtheorem{lem}[thm]{Lemma}

\theoremstyle{definition}
\newtheorem{defn}[thm]{Definition}
\newtheorem{constr}[thm]{Construction}

\theoremstyle{definition}
\newcounter{hypc}
\newtheorem{quest}[hypc]{Question}
\newtheorem{hypot}[hypc]{Hypothesis}

\theoremstyle{remark}
\newtheorem{rem}[thm]{Remark}

\def\WCl{{\mathrm{WCl}}}
\def\uT{{\hat T}}
\def\LP{{\mathscr L}}
\def\I{{\mathbb I}}
\def\tP{{\Tilde P}}
\def\W#1{\hbox{\tt #1}}
\def\divides{\mathrel{\hbox{$\,|\,$}}}
\def\Pz{P^{\Set{Z}}}

\begin{document}

\title[The limit polynomials for Chacon transformation]
	{Several questions and hypotheses concerning the limit polynomials for Chacon$_{(3)}$~transformation}

\if0=1
\author{A.A.\,Prikhod'ko}
\email{sasha.prihodko@gmail.com}
\altaffiliation{Laboratoire~des~Mat\'{e}matiques Rafa\"{e}l Salem, l'Universit\'{e} de Rouen}
\thanks{CNRS (Normandie, France)}
\affiliation{Moscow State University, Dept.\ of Mechanics and Mathematics}

\author{V.V.\,Ryzhikov}
\email{vryzh@mail.ru}
\affiliation{Moscow State University, Dept.\ of Mechanics and Mathematics}
\fi

\author{A.A.\,Prikhod'ko$^{1}$ and V.V.\,Ryzhikov$^{2}$}
\affiliation
{$^{1}$Moscow State University, Dept.\ of Mechanics and Mathematics\ / 
Laboratoire~des~Mat\'{e}matiques Rafa\"{e}l Salem, l'Universit\'{e} de Rouen. 
E-mail:~\href{sasha.prihodko@gmail.com}{sasha.prihodko@gmail.com}
\\
$^{2}$Moscow State University, Dept.\ of Mechanics and Mathematics. 
E-mail:~\href{vryzh@mail.ru}{vryzh@mail.ru}}

\date{\today}

\begin{abstract}
We study the weak closure $\LP = \WCl(\{\uT^k\})$ of powers of 
non-singular Chacon transformation with $2$-cuts. % with three cuts. 
It~is still an open question does $\LP$ contain any Markov operator 
except an orthogonal projector to the constants $\Theta$ and some polynomials $P(\uT)\:$? 
In~this~paper we calculate a~particular set of limit polynomials 
$$
	P_m(\uT) = \lim_{n \to \infty} \uT^{-mh_n}, \qquad m \in \Set{Z}, 
$$
where $h_n = (3^n-1)/2$ are the sequence of heights of towers 
in a standard rank one representation of the Chacon map. 
%% Representing $P_m(\uT)$ in the form ${P_m(\uT) = \uT^{l_m}\,\Tilde P_m(\uT)}$ with largest possible~$l_m$, 
We~show that for any ${d \ge 2}$ 
the family of limit polynomials contains infinitely many %%%an irreducible 
polynomials of degree~$d$.  
We also formulate hyposeses and open questions concerning the sequence $P_m$ and the~entire set~$\LP$. 
\end{abstract}

\maketitle

\section{Introduction}

Chacon transformation in terms of symbolic dynamics can be definied 
as a substitution system over the finite alphabet ${\Set{A} = \{0,1\}}$ 
via a~pair of substitution rules 
$$
	0 \; \mapsto \; 0010, \qquad 
	1 \; \mapsto \; 1. 
$$
Starting with an initial word $w_0 = 0$ and applying the substitution transform 
we construct the sequence $w_n$, 
\begin{gather*}
	w_0 = 0 \\ 
	w_1 = 0010 \\ 
	w_2 = 0010001010010 \\ 
	w_3 = {\scriptstyle 0010001010010001000101001010010001010010} \\ 
	\dots 
\end{gather*}
and then define an infinite word $w_\infty$ such that each $w_n$ is a prefix of~$w_\infty$. 
Further, considering the closure $X$ of all shifts of~$w_\infty$ in the space $\Set{A}^\infty$ 
endowed with the Tikhonov topology 
we come to a~topological dynamical system $(S,X,\cB)$, where 
$\cB$ is the $\sigma$-algebra of Borel sets and $T$ is the shift transformation, 
\begin{equation*}
	T \Maps \dots, x_0,x_1,\dots,x_j, \dots \; \mapsto \; \dots, x_1,x_2,\dots,x_{j+1}, \dots
\end{equation*}
Let us consider a~natural invariant measure~$\mu$ on the measurable space $(X,\cB)$ 
defined as follows. 
For a~finite word $w$ let $\mu([w])$ be the empirical probability of 
observing $w$ in~$w_\infty$, where % $[w]$ is the set encoded by~$w$, 
where $[w]$ is the set encoded by~$w$: 
$$
	[w] \eqdef \{ x\in X \where x_0 = w(0), \dots, x_{\ell-1} = w(\ell-1) \}, 
$$
%where 
$\ell = |w|$ is the length of~$w$ and $w(j)$ denotes the letter  at position~$j$  in~$w$. 

\begin{defn}
The map $T$ considered as a~measure-pre\-ser\-ving invertible transformation 
of the probability space $(X,\cB,\mu)$ 
is called {\it non-singular Chacon transformation with $2$-cuts\/} 
or {\it Chacon}$_{(3)}$ {\it transformation\/} 
(see~\cite{ChaconMap,FriedmanErgTh}). 
\end{defn}

Transformation $T$ has an interesting combination 
of ergodic properties. It is known to be weakly mixing and power weakly mixing~\cite{DanilenkoOnChaconPWMix}, 
but not strongly mixing~\cite{ChaconMap}. 
It~has trivial centralizer \cite{delJuncoOnASimple} and minimal self-joinings~\cite{JRS}. 
It is also known that the spectral measure $\sigma$ of Chacon transformation $T$  
is singular and its convolutions satisfy the~following condition of pairwise singularity~\cite{PR}, 
\begin{gather*}
	\sigma \perp \sigma \conv \sigma, \\ 
	\sigma \conv \sigma \perp \sigma \conv \sigma \conv \sigma, \\
	\dots \\ 
	\sigma^{*k} \perp \sigma^{*\ell} \qquad \text{for any} \quad k \not= \ell. 
\end{gather*}
The study of convolutions of the spectral type measure $\sigma$ goes back to the 
Kolmogorov's question concerning the hypothetic group property of spectrum: 
{\it is it true that $\sigma \conv \sigma \ll \sigma$\/}? 
This property holds for the discrete part of spectrum, but it is false for the singular component. 
Moreover, now we know many examples of ergodic transformations~$T$ 
such that ${\sigma \conv \sigma \perp \sigma}$ 
(see \cite{Oseledec2,StepinGenPropr,GoodsonSpThepry,delJuncoLem}).   

For a survey of problems in modern spectral theory of dynamical systems the~reader 
can refer to \cite{LemEncycloSpTh} and~\cite{KatokThouvenotSpTh}. 

\begin{defn}
We say that a~map $T$ is {\it mixing\/} if 
$$
	\mu(T^k A \cap B) \to \mu(A)\,\mu(b) \quad \text{as \ $k \to \infty$}, 
$$
for any measurable sets $A$ and $B$, and we call $T$ {\it weakly mixing\/} 
if the convergence holds for a~subsequence~$k_j$. 
\end{defn}

Both mixing and weak mixing properties can be described in spectral terms. 

\begin{defn}
Let $\uT$ be the unitary {\it Koopman operator}, associated with~$T$ 
and acting in the~separable Hilbert space ${H = L^2(X,\mu)}$ by the following rule 
$$
	\uT \Maps f(x) \mapsto f(Tx). 
$$
\end{defn}

A sequence of bounded linear operators $\cA_j \Maps H \to H$ 
in a Hilbert space~$H$ 
{\it converges weakly\/} to $\cA$ if for any ${f,g \in H}$ 
$$
	\scpr<\cA_jf,g> \to \scpr<\cA f,g>, \qquad j \to \infty. 
$$
Let $\Theta$ denote the orto-projector to constants, 
$$
	(\Theta f)(x) \equiv \int_X f(z) \,d\mu(z). 
$$
%%%%% and let $\I$ be the identity operator. 
% 
A transformation $T$ is weakly mixing if and only if  
$$
	T^{k_j} \to \Theta 
$$
for some subsequence $k_j$. It means that $\Theta$ is in the weak closure $\LP = \WCl(\{\uT^k\})$ 
of powers~$\uT^k$. % Furthermore, 

%=================================================================================================
\section{Limit polynomials}
 
In our investigation~\cite{PR} to prove the pairwise singularity 
of the convolutions $\sigma^{*k}$ we used the following observation. 

\begin{lem}\label{lemQuadratic} 
In the weak close of powers $\LP = \WCl(\{\uT^k\})$ 
for Chacon transformation $T$ one can find 
an infinite family of non-trivial square polynomials 
$$
	Q_m(\uT) = \frac{(3^s-1)\I + 2(3^s+1)\uT + (3^s-1)\uT^2}{4\cdot 3^s}, 
$$
for $m = 3^s + 1$ and, moreover, 
$$
	Q_m(\uT) = \lim_{n \to \infty} \uT^{mh_n-l_s}, 
$$
where $l_s = (3^s-1)/2$ and $\I$ is the identity operator.  
%$$
%	\I \Maps f(x) \to f(x). 
%$$
\end{lem}

In order to understand this phenomem let us consider a~simplier case. 

\begin{lem}\label{lemOneHalf}
There exists a sequence ${k_j \to \infty}$ such that 
$$
	\uT^{k_j} \to \frac{\I + \uT}{2}. 
$$
\end{lem}

\begin{proof}
Another way to define Chacon transformation as a~measure-preserving transformation 
is to use the concept of rank one transformation. 

\begin{defn}
Let $T$ be a measure-preserving transformation of a probability space $(X,\cB,\mu)$. 
Then $T$ is called {\it rank one transformation\/} 
if there exists a sequence of Rokhlin tower partitions 
$$
	\xi_j = \{B_j,TB_j,T^2B_j,\dots,T^{h_n-1}B_j,\: E_j\}
$$
of the phase space such that ${\mu(E_j) \to 0}$ and for any measurable set $A$ 
one can fins $\xi_j$-measurabe sets $A_j$ approximating $A$: 
$
	\mu(A_j \syms A) \to 0 
$
as $j \to \infty$. 
\end{defn}
 
In fact, Chacon transformation is rank one and can be constructed using 
so-called cutting-and-stacking construction. 

\begin{figure}[ht]
	\includegraphics[width=60mm]{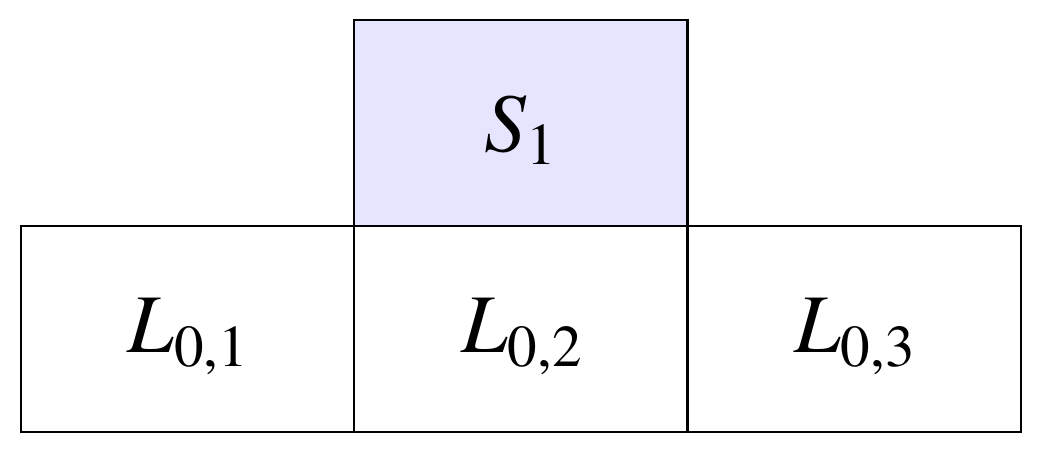} 
	\caption{Chacon$_{(3)}$ transformation: several steps in 
		the cutting-and-stacking construction:~${n = 1}$  
	} 
	\label{fChacomMapCS1}
\end{figure}

\begin{figure}[ht]
	\includegraphics[width=60mm]{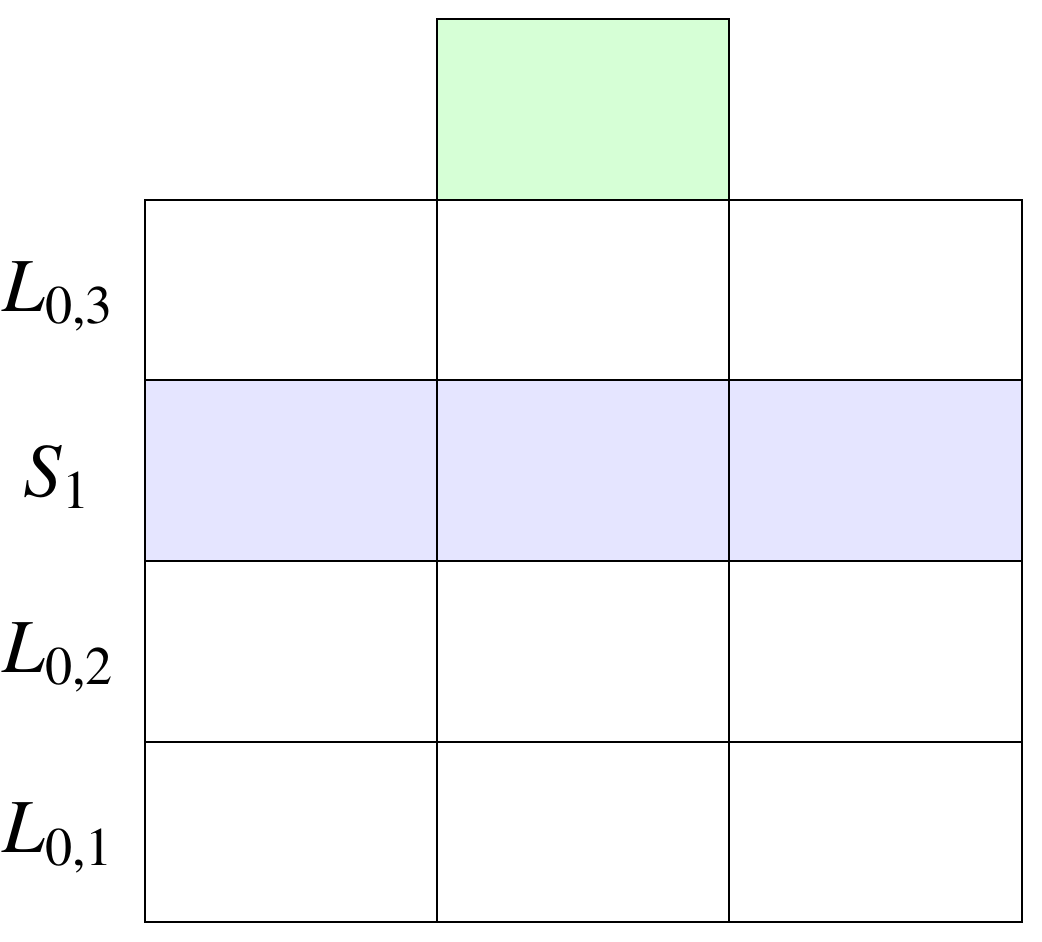} 
	\caption{Cutting-and-stacking construction: ${n = 2}$} 
	\label{fChacomMapCS2}
\end{figure}

\begin{figure}[ht]
	\begin{tabular}{ccc}
		\includegraphics[height=84mm]{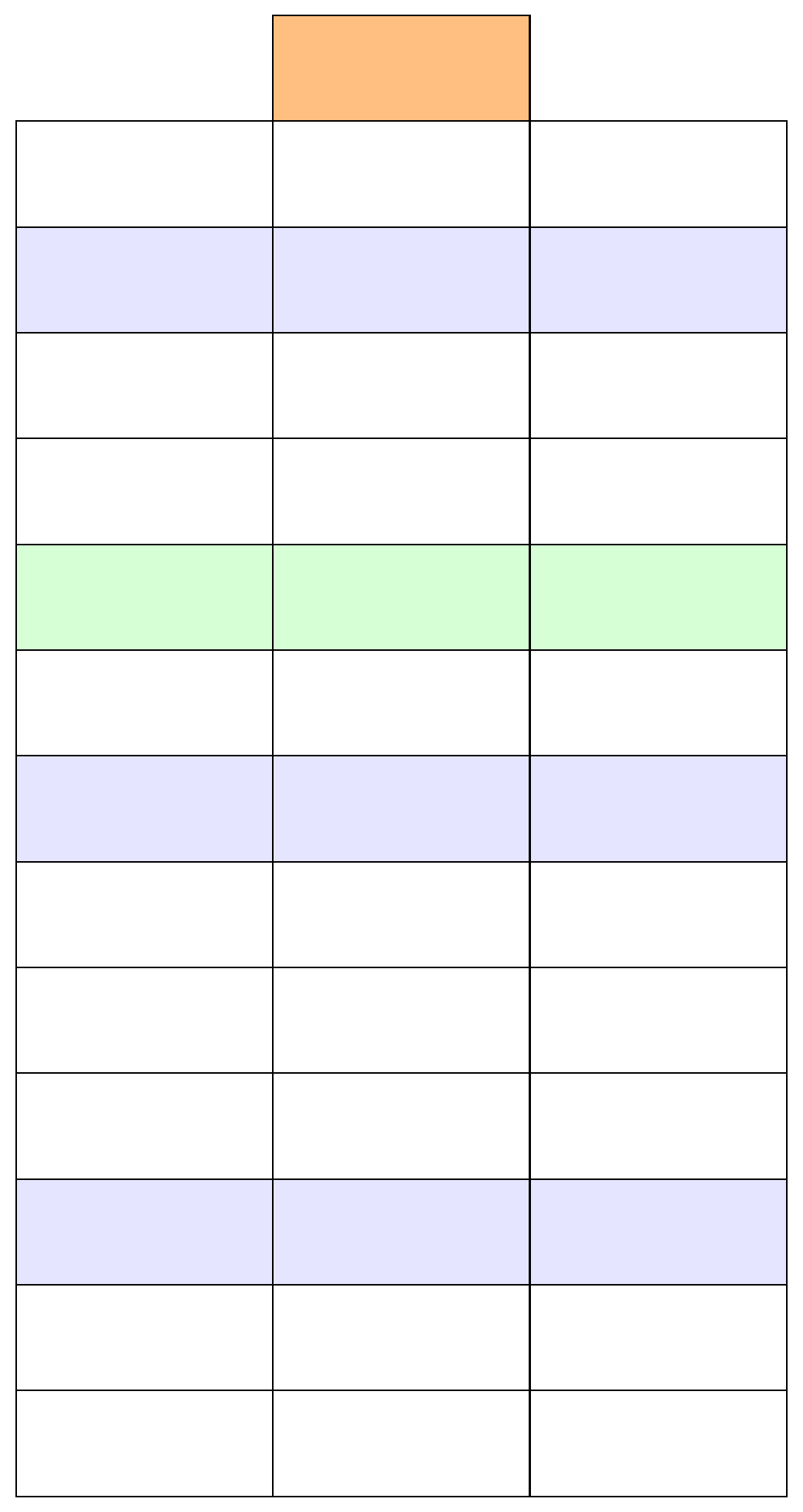} & \lower-33pt\hbox{\large $\;\; \longrightarrow \;\;$} & 
		\includegraphics[height=84mm]{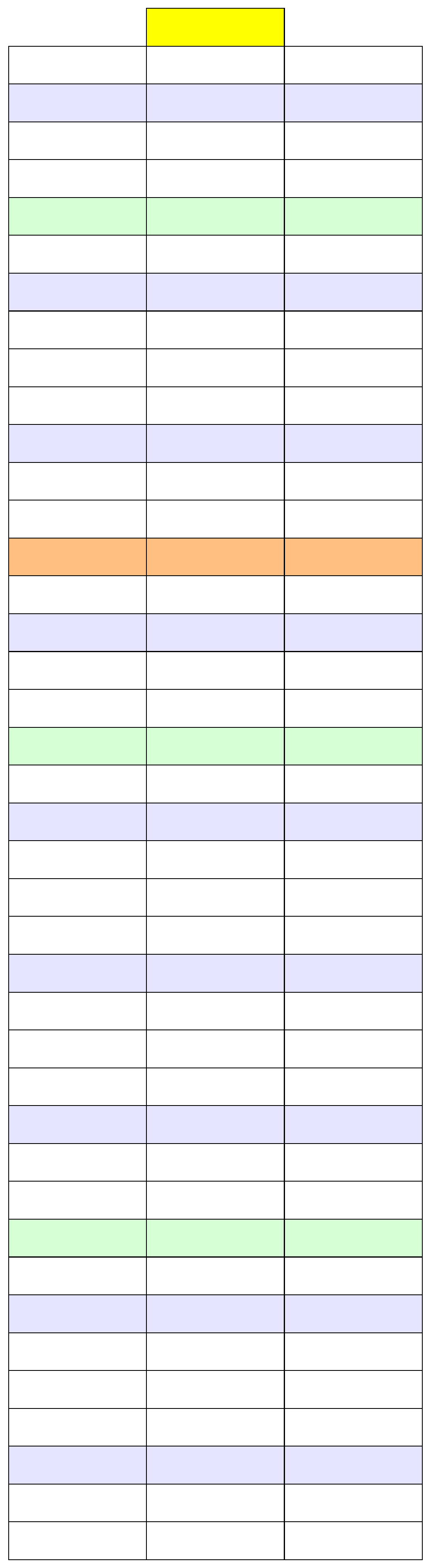}
	\end{tabular}
	\caption{Cutting-and-stacking construction: ${n = 3}$ and ${n = 4}$} 
	\label{fChacomMapCS3}
\end{figure}

\begin{constr}
We start with a unit segment $[0,1]$ interpreted as a Rokhlin tower $U_0$ of height ${h_0 = 1}$. 
% Set ${U_0 = [0,1]}$. 
Then we cut this segment twice, in three equal parts 
$$
	L_{1,0} = [0,1/3), \quad L_{1,1} = [1/3,2/3), \quad L_{1,2} = [2/3,1], 
$$
and add one additional ``level'', a segment $S_{1}$ of length $1/3$ which is drawn 
above the middle part $[1/3,2/3)$ (see~fig.~\ref{fChacomMapCS1}),  
\begin{equation*}
	\begin{alignedat}{3}
		&         & \quad & S_1     & \quad & \\ 
		& L_{1,0} & \quad & L_{1,1} & \quad & L_{1,2}
	\end{alignedat} 
\end{equation*}
Now we stack all these segments in the natural order: 
$L_{1,0}\,L_{1,1}\,S_{1}\,L_{1,2}$ and we get the next Rokhlin tower $U_1$ of height ${h_1 = 4}$ 
(see~fig.~\ref{fChacomMapCS2}). 
In other words, we assume that 
$$
	L_{1,0} \stackrel{T}\longrightarrow 
	L_{1,1} \stackrel{T}\longrightarrow 
	S_1     \stackrel{T}\longrightarrow 
	L_{1,2}, 
$$
and $T$ will be defined on $L_{1,2}$ on the next steps of the construction. 
We repeat the same procedure with the new tower: we cat it in three equal columns, 
put one additional level to the top of the middle column and stack together 
(fig.~\ref{fChacomMapCS1}--\ref{fChacomMapCS1}). 
 
At~each step of the construction we have a Rokhlin tower $U_n$ of height $h_n = (3^n-1)/2$. 
It can be easily checked that this sequence serves as an approximating sequence 
of Rokhlin towers in the definition of rank one transformation. 
\end{constr}
 
Note that if we draw all the additional level above the corresponding subcolumns 
without restacking the tower $U_n$ at the step~$n$ 
we come to the following representation of the Chacon map (see~fig.~\ref{fChacomMapCSTogether}). 

\begin{figure}[ht]
	\includegraphics[width=72mm]{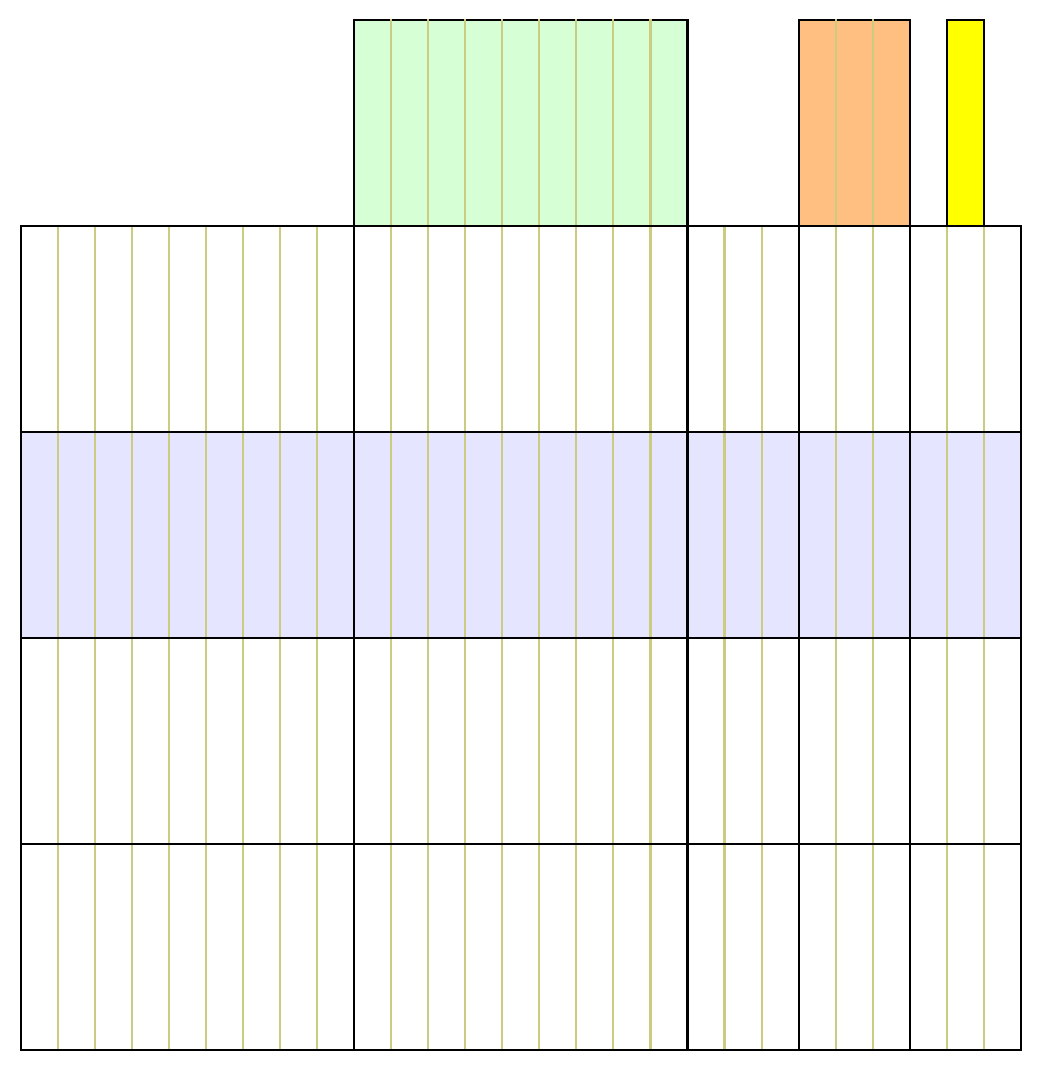} 
	\caption{Chacon$_{(3)}$ transformation without restacking} 
	\label{fChacomMapCSTogether}
\end{figure}

\begin{constr}
Let us consider the compact group of $3$-adic ingtegers $\G = \Set{Z}_{(3)}$. 
We associate $\G$ with the set of one-sided $3$-adic sequences 
$$
	y = (y_1,y_2,\dots,y_k,\dots), \qquad y_k \in \{0,1,2\}. 
$$
As a measure space $\G$ is isomorphic to the unit segment $[0,1]$ 
by the mapping 
$$
	y \mapsto \sum_{k=1}^\infty \frac{y_k}{3^k}. 
$$
It follows easily from the cutting-and-stacking construction that 
Chacon map $T$ is the integral transformation 
over the adding machine transformation 
$$
	S \Maps \G \to \G \Maps y \to y + 1 
$$
acting on the base level of the tower $U_n$ identified with $\G$ 
with the ceiling function ${r_n(y) = h_n + \phi_0(y)}$ (see~fig.~\ref{fChacomMap}), % where 
\begin{equation*}
	\phi_0(y) = \left\{ \begin{split}
		& 0, & & \text{if $y = \W{22}\dots\W{20*}$} \\ 
		& 1, & & \text{if $y = \W{22}\dots\W{21*}$} 
	\end{split} \right. 
\end{equation*}
where \W{*} indicates any symbol in alphabet $\{0,1,2\}$ if put inside a~block, 
and any infinite sequence of symbols if it ends the block. 
\end{constr}

\begin{figure}[ht]
	\includegraphics[width=75mm]{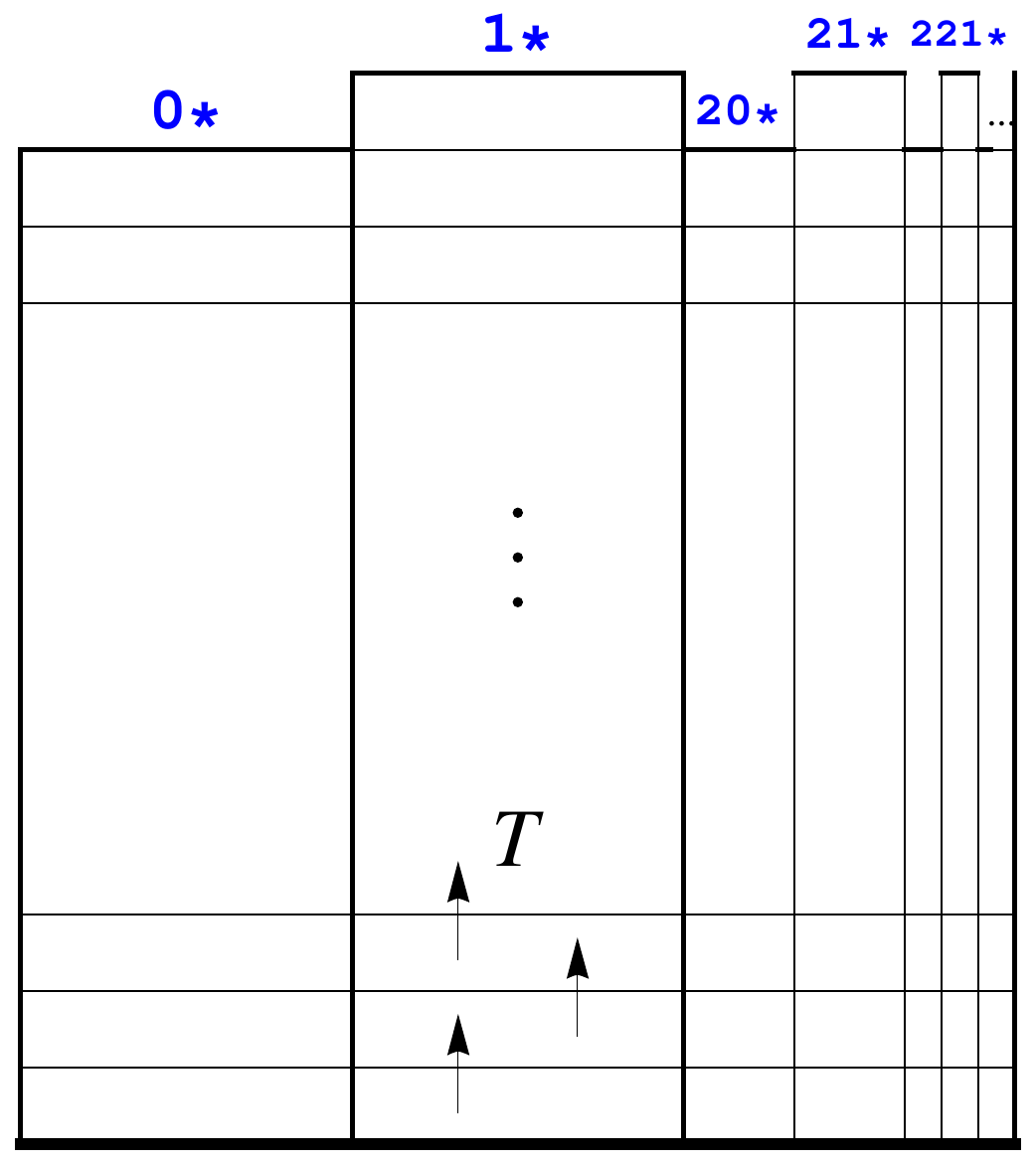} 
	\caption{Chacon$_{(3)}$ transformation: cutting-and-stacking construction and cocycle $\phi_0(y)$.}  
	\label{fChacomMap}
\end{figure}
 
Now we are ready to finish the proof of lemma~\ref{lemOneHalf}. 
It~can~be easily checked that any measurable function $f \in L^2(X,\mu)$ 
is approximated by functions constant on levels of a~tower in the rank one representation. 
So, assume that $f$ is constant on the levels of~$U_n$. 
Let us partition $U_n$ 
into sets $U_n^{(0)}$ and $U_n^{(1)}$ according to the value of the cocycle~$\phi_0(y)$, 
where $y$ is considered as a point in the base of~$U_n$. 
We see that 
$$
	f(T^{h_n}x) = f(x), \quad \text{if} \quad x \in U_n^{(0)}
$$
and
$$
	f(T^{h_n}x) = f(T^{-1}x), \quad \text{if} \quad x \in U_n^{(1)}
$$
for all points ${x \in U_n}$ except the first level $B_n$ of the tower~$U_n$ 
(observe that ${\mu(B_n) \to 0}$). 
Thus, 
$$
	\uT^{h_n} \to \frac{\I + \uT^{-1}}2, 
$$
since $\mu(U_n^{(0)}) = \mu(U_n^{(0)})$,
and applying conjugation we complete the proof. 
\end{proof} 

Analyzing the effects used in the proof we see that lemma~\ref{lemOneHalf} 
can be easily extended in the following way. 
Given $m \in \Set{N}$ let us consider the sum 
$$
	\phi_0^{(m)} = \phi_0(y) + \phi_0(Sy) + \dots \phi_0(S^{m-1}y) 
$$
and define the corresponding distribution $\rho_m$ of the values of~$\phi_0^{(m)}$. 
Actually $\rho_m$ is the measure on~$\Set{Z}$ with a~finite support 
which is the image by~$\phi_0^{(m)}$ of 
the Haar probability measure on~$\G$. 

\begin{lem}
For any ${m \in \Set{N}}$ the sequence $\uT^{-mh_n}$ converges weakly 
to a polynomial depending on~$\uT$, and 
$$
	P_m(\uT) \eqdef \lim_{n \to \infty} \uT^{-mh_n} = \int_{\Set{Z}} \uT^k \,d\rho_m(k). 
$$
\end{lem}

The scheme of the proof can be found in~\cite{PR}, and the idea can be explained as follows. 
Passing the tower $U_n\:$ $m$~times we count (in~addition to~$mh_n$) the values 
of the cocycle~$\phi_0(y)$ at the points 
$$
	\phi_0(y), \quad \phi_0(Sy), \quad \dots \dots \phi_0(S^{m-1}y). 
$$

Let us consider several first polynomials $P_n(\uT)$: 
\begin{gather*}
	P_1(\uT) = \frac12(\I + \uT) \\ 
	P_2(\uT) = \frac16(\I + 4\uT + \uT^2) \\ 
	P_3(\uT) = \frac12(\uT + \uT^2) \\ 
	P_4(\uT) = \frac19(2\uT + 5\uT^2+2\uT^3) \\ 
	P_5(\uT) = \frac1{18}(\uT + 8\uT^2+8\uT^3+\uT^4) 
\end{gather*}
Since the weak closure $\WCl(\{\uT^j\})$ is invariant under multiplication by $\uT^s$ 
for any ${s \in \Set{Z}}$ we can reduce the polynomials $P_m(\uT)$ by 
the smallest power $l_m$ of~$\uT$ in~$P_m(\uT)$. Set 
$$
	\tP_m(z) = z^{-l_m} \cdot P_m(z). 
$$
Let us represent $\tP_m(z)$ in the form 
$$
	\tP_m(z) = a_{m,0} + a_{m,1} z + \dots + a_{m,d(m)} z^{d(m)}, 
$$
where $d(m) = \deg \tP_m(z)$. 

\begin{lem}
The coefficients $a_{m,j} \in \Set{Q}$ satisfy the following Markov property 
$$
	\sum_{j=0}^{d(m)} a_{m,j} = 1, \quad \text{and} \quad a_{m,j} \ge 0.
$$
\end{lem}

% \medskip

\begin{figure}[ht]
	\includegraphics[width=75mm]{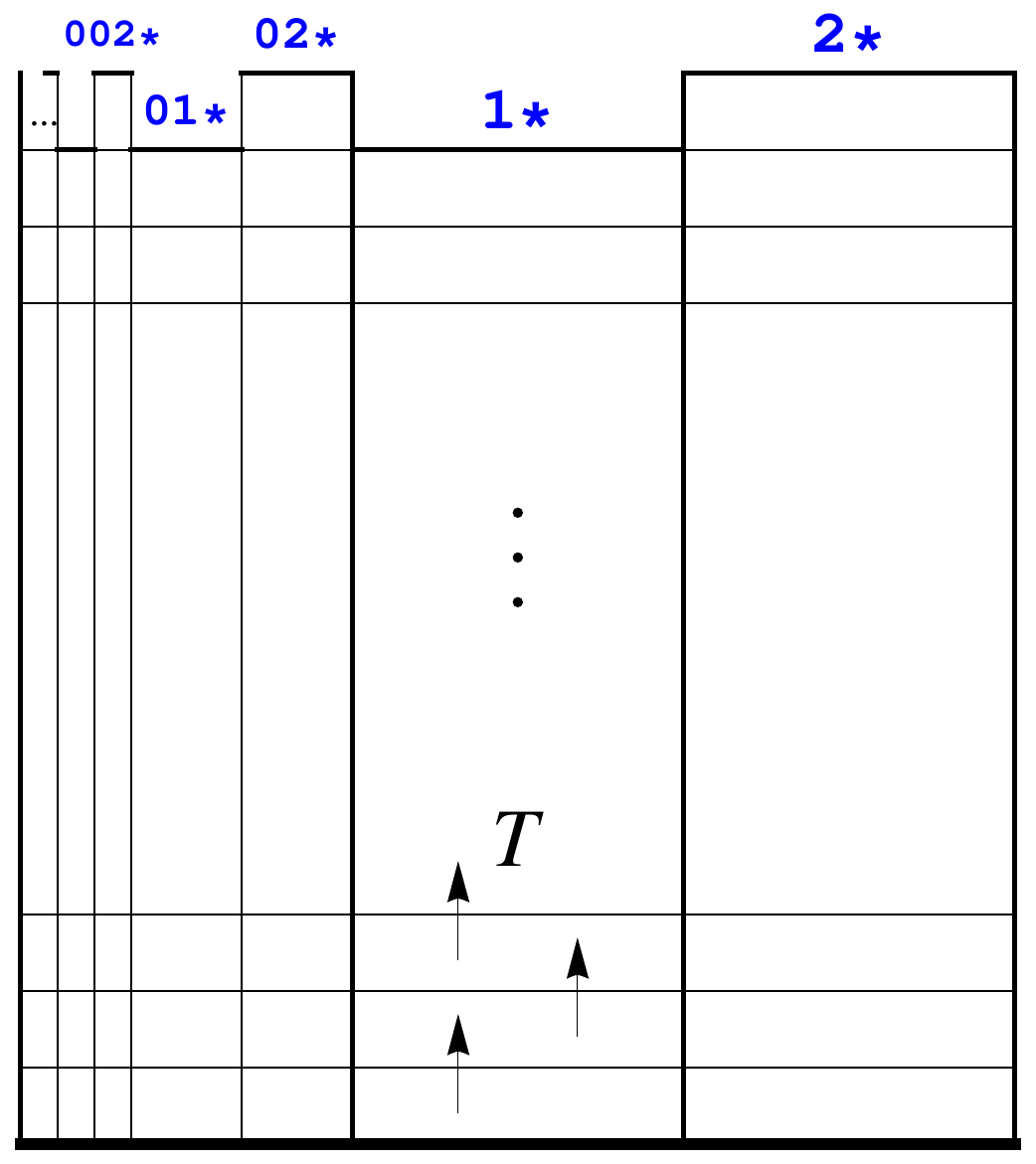} 
	\caption{Chacon$_{(3)}$ transformation after the coordinate change ${y \mapsto y+1}$ 
		and cocycle $\phi(y)$.} 
	\label{fChacomMapMod}
\end{figure}

% \subsection{Calculating the polynomials $\Tilde P_m(\uT)$} 

In~table~1 of the Appendix we list the first $365$ polynomials $\tP_m(z)$. 

Let us discuss several remarks explaining the structure of this table. 
% 
% While calculating the polynomials $P_m(z)$ 
First, for simplicity of calculations 
we apply the transform ${y \mapsto y+}$ to the base of the tower $U_n$ and consider 
the following cocycle $\phi(y)$ insteed of~$\phi_0(y)$ (see~fig.~\ref{fChacomMapMod}),  
\begin{equation*}
	\phi(y) = \left\{ \begin{split}
		& 0, & & \text{if $y = \W{00}\dots\W{01*}$} \\ 
		& 1, & & \text{if $y = \W{00}\dots\W{02*}$} 
	\end{split} \right. 
\end{equation*} 
The function $\phi(y)$ is more convenient for calculation of the iterates $\phi(S^k y)$. 
 
\begin{lem}
For any power $3^\ell$ of three we have 
\begin{equation*}
	\phi^{(3^\ell)}(y) = \left\{ \begin{split}
		& 0, & & \text{if $y = \W{*}^\ell(\W{0})\W{1*}$} \\ 
		& 1, & & \text{if $y = \W{*}^\ell(\W{0})\W{1*}$} 
	\end{split} \right. 
\end{equation*} 
where the notation $(\W{0})$ is used for any sequense of zeros (including empty sequence), 
and $\W{*}^\ell$ denotes an arbitrary word of length~$\ell$. 
An equivalent way how we can state this property of the cocycle is to say that 
$$
	\phi^{(3^\ell)}(y) = \phi(A y), 
$$
where $A$ is the non-invertible left shift: 
$$
	A(y_1y_2\cdots y_k\cdots) = y_2 y_3 \cdots y_{k+1} \cdots
$$
\end{lem}

It follows immediately from this lemma that polynomials $\tP_m(z)$ repeat 
after multiplication by~$3$, 
$$
	P_{3m}(z) = P_m(z). 
$$

\begin{thm}\label{thmExist}
For any $d \in \Set{N}$ the family $\{\tP_m\}$ contains infinitely many polynomials of degree~$d$. 
\end{thm}

\begin{proof}
The theorem is based on the following observation. If we consider 
functions $\phi^{(m)}(y)$ and $\phi(S^k y)$ as random variables defined on~$\G$, then 
$\phi^{(m)}(y)$ and $\phi(y)$ are almost independent. 
Thus, extending the proof of lemma~\ref{lemQuadratic} we see that 
configurations 
$$
	m(\ell_1,\dots,\ell_{d-1}) = \W{10}^{\ell_1}\W{10}^{\ell_2}\W{100}\dots\W{0}^{\ell_{d-1}}\W{1}_{3} 
$$
generates for sufficiently big $\ell_j$ polynomials $P_{(\ell_j)}(z)$ of degree~$d$ such that 
$$
	\lim_{\ell_j \to \infty} P_{(\ell_j)}(\uT) = \frac1{2^d}(\I + \uT)^d. 
$$
\end{proof}

To illustrate the construction used in the proof let us consider configuration 
$$
	m(\ell_1,\ell_2) = \W{10}^{\ell_1}\W{10}^{\ell_2}\W{1}_3 = 
	\W{1} \overbrace{\W{00}\dots\W{0}}^{\ell_1} \W{1} \overbrace{\W{000}\dots\W{0}}^{\ell_2} \W{1}_3
$$
Set
$$
	p^{-[i,j]} = p^{-i} + p^{-i-1} + \dots + p^{-j}, 
$$
and notice that $3^{-[1,\infty]} = 1/2$. 

\begin{lem} 
$\cP_{m(\ell_1,\ell_2)}$ is a self-reciprocal polynomial, 
$$
	\cP_{m(\ell_1,\ell_2)} = \gamma z^3 + (1/2 - \gamma)z^2 + (1/2 - \gamma)z + \gamma, 
$$
where
$$
	\gamma = 3^{-[1,\ell_1]} 3^{-[1,\ell_2]} + 
		3^{-[1,\ell_1]} 3^{-(\ell_2+1)} + 3^{-(\ell_1+1)} 3^{-[1,\ell_2]}. 
$$
\end{lem}

\begin{proof}
The proof of this lemma is very close to that of lemma~\ref{lemQuadratic}. 
\end{proof}

In the next section we formulate a set of hypotheses concerning the properties 
of the limit polynomials $\tP_m(z)$. In~hypothesis~\ref{H1} we conjecture that 
all polynomials $\tP_m(z)$ are self-reciprocal, i.e.\ they have coefficients 
symmetric under the transform ${k \mapsto d-1}$, ${d = \deg \tP_m}$, 
$$
	\tP_m(z) = \sum_{j=0}^{d(m)} a_{m,j} z^j, \qquad a_{m,j} = a_{m,d(m)-j}. 
$$
In other words, the sequence $a_{m,j}$ is symmetric. 
% 
% Denote: $d(m) = \deg \tP_m$. 
 
\begin{lem}
If hypothesis~\ref{H1} holds then $d(m) \in 2\Set{Z}+1$ implies that $(-1)$ is a~root of $\tP_m$. 
In~particular, any polynomial $\tP_m$ of odd degree is factorized, 
$$
	\cP_m(z) = (z+1)\, R_m(z). 
$$ 
\end{lem}

The proof is a simple calculation. Nevertheless, we can ask a~question: 
is it the only way to factorize $\tP_m$? 

\begin{thm}\label{thmIrrCubic} 
The family of limit polynomials $\cP_m(z)$ contains infinitely many 
cubic polynomials for which ${R_m(z) = (z+1)^{-1}\tP_m(z)}$ are irreducible. 
\end{thm}

\begin{proof}
Indeed, consider cubic polynomials given by configurations $\W{10}^{\ell_1}\W{10}^{\ell_2}\W{1}$ 
with $\ell_1 = \ell_2$ (see~table~2 of the Appendix). With a~simplified notation ${\ell = \ell_1}$ we have 
\begin{multline*}
	\cP_{m(\ell,\ell)}(z) = \\ = 
		\frac{(3a^2 + 2a)(z^3+1) + (3^{2\ell+1} - 3a^2 - 2a))(z^2+z)}
			{2\cdot 3^{2\ell+1}}, 
\end{multline*}
where
$$
	3^{-[1,\ell]} = \frac13 + \dots + \frac1{3^\ell} = \frac{a}{3^\ell}, \qquad \gcd(a,3) = 1. 
$$
Next, let us apply the transform ${z = -1+w}$ to~$\cP_{m(\ell,\ell)}$. We get a new polynomial 
$$
	P^*(w) = (3a^2+2a)w^3 + (3^{2\ell+1}-8a-12a^2)(w-1). 
$$
What are the common divisors of 
$$
	X = 3a^2+2a \quad \text{and} \quad Y = 3^{2\ell+1}-8a-12a^2 \:?
$$
We have
$$
	Y + 4X = 3^{2\ell+1}
$$
and, at the same time, 
$$
	X = a(3a+2)
$$
is factorized in two numbers, both are relatively prime to~$3$. 
Thus, taking any prime divisor of $Y$ and applying Eisenstein's criterion we see that 
$P^*(w)$ is irreducible over~$\Set{Q}$. 
\end{proof}

\begin{lem}[Eisenstein's criterion]
Consider a polynomial $P \in \Set{Q}[z]$, 
$$
	P(z) = a_n z^n + \dots + a_1 z + a_0, 
$$
and suppose that there exists a prime number $p$ sush that 
\begin{gather*}
	p \not\divides a_n, \qquad p^2 \not\divides a_0, \\
	p \divides a_j \quad \text{for} \quad j=0,1,\dots,n-1. 
\end{gather*}
Then $P(z)$ is irreducible over~$\Set{Q}$. 
\end{lem}

It is interesting to remark that the quadratic polynomials given in lemma~\ref{lemQuadratic} 
are factorized over~$\Set{Q}$, thus, to see that there exists irreducible polynomial 
$\tP_m(z)$ we have to consider a particular example: 
$$
	\tP_2(z) = \frac16(z^2+4z+1). 
$$
Letting $z=-1+w$ we get a~polynomial 
$$
	P^*(w) = z^2+2z-2. 
$$
We can apply Eisenstein's criterion to $P^*$, since $2$ divides all the coefficient except 
the coefficient in~$z^2$, and $4$ do not divide~$-2$.

%=================================================================================================
\section{Questions and hypotheses}

%%%%%%%%%%%%%%%%%%%%%%%%%%%%%%%%%%%%%%%%%%%%%%%%%%%%%%%%%%
% Basic properties
%%%%%%%%%%%%%%%%%%%%%%%%%%%%%%%%%%%%%%%%%%%%%%%%%%%%%%%%%%

% Let $d(m) = \deg \tP_m(z)$. 

\begin{hypot}\label{H1} 
The limit polynomials $\tP_m(z)$ are {\it self-re\-ci\-pro\-cal}, that is 
$$
	\tP_m(z) = \sum_{k=0}^{d(m)} a_k z^k, \qquad a_k = a_{d(m)-k}. 
$$
\end{hypot}

{\it Corollary.} If hypothesis~\ref{H1} is true then %, evidently, 
$-1$ is a~root of a~polynomial $\tP_m(z)$, whenever $d(m) \in 2\Set{Z}+1$. 

\medskip 
We~have to mention that most questions below presume or at least require 
hypothethis~\ref{H1} for a~particular~$m$. 

\begin{defn}
Consider two configurations of the same length 
$$
	c_1 c_2 \dots c_N \quad \text{and} \quad c'_1 c'_2 \dots c'_N, 
$$
with $c_j,c'_j \in \{0,1,2\}$. We say that $m'$ is {\it conjugate\/} to $m$ 
and write $m' = m^*$
if ${c'_{j} = c_{N+1-j}}$. 
\end{defn}

\begin{hypot}\label{HSym}
The polynomials $\tP_m(z)$ and $\tP_{m^*}(z)$ coincide for 
any pair of conjugate configurations $m$ and~$m^*$. 
\end{hypot}

It can be observed from table~$1$ that some polynomials coincide 
even for non-conjugate configurations, for example, 
for ${m = 10 = \W{101}_3}$ and ${m' = 26 = \W{222}_3}$. 
%$$
%$$

\begin{quest}
Which pairs of polynomials $\tP_m(z)$ and $\tP_{m'}(z)$ coincide?
\end{quest}

Let $|m|_3$ be the length of the $3$-adic expansion of~$m$ if ${3 \not\divides m}$, 
and ${|m|_3 = |3^{-1}m|_3}$ otherwise. 

\begin{hypot}
$\Pz_m(z) = 2\cdot3^{|m|_3}\cdot\tP_m(z)$ is a polynomial with integer coefficients, 
$$
	\Pz_m = b_{m,0} + b_{m,1}z + \dots + b_{m,d(m)}z^d. 
$$
The greatest common divisor of $b_{m,j}$ is $1$ or~$2$. 
\end{hypot}

%%%%%%%%%%%%%%%%%%%%%%%%%%%%%%%%%%%%%%%%%%%%%%%%%%%%%%%%%%
% Irreducibility and Roots 
%%%%%%%%%%%%%%%%%%%%%%%%%%%%%%%%%%%%%%%%%%%%%%%%%%%%%%%%%%

% {thmIrrCubic}
This is a well-known fact that the set of all weak limits of powers $\LP$ 
is a~semigroup. Thus, it is a~natural question: can we get a polynomial $P_m(z)$ 
as a~product of two different elements of~$\LP$?
 
\begin{hypot}\label{hypIrr}
The polinomial $\tP_m(z)$ has two or more factors which are not $(z+1)$ if and only if 
(see table~3)
$$
	|m|_3 \in 2\Set{Z} \quad \text{and} \quad m = m^*. 
$$ 
\end{hypot}

For example, the polynomial 
$$
	\tP_{68}(z) = \frac1{81}(3+5z+z^2)(1+5z+3z^2)
$$
corresponds to a symmetric configuration $68 = \W{2112}_3$. 

\medskip
In particular, if hypothesis~\ref{hypIrr} is true then the roots $r_j$ 
of a polinomial $\Pz_m(z)$ starting with $z^d + \dots$ are algebraic integers. 

\begin{figure}[ht]
	\includegraphics[width=75mm]{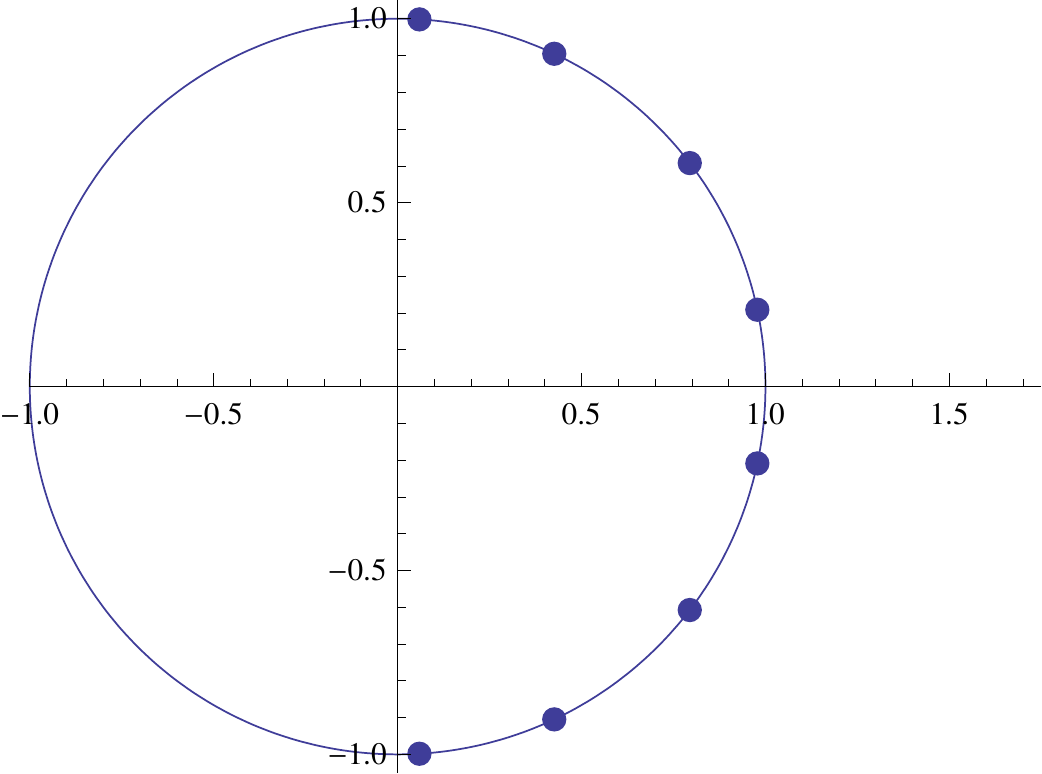} 
	\caption{Roots of the polynomials $Q_{1094}$.} 
	\label{fRootsA}
\end{figure}

\begin{hypot}\label{hypReal}
All roots of any $\tP_m(z)$ are real numbers ({\it Lee--Yang property\/}). 
\end{hypot}

\begin{rem}
It follows directly from hypothesis~\ref{H1} as well as the definition of the polynomials~$\tP_m(z)$ 
that the roots of $\tP_m(z)$ must be negative, and they appear in pairs: $r$~and~$r^{-1}$. 
\end{rem}

\begin{figure}[ht]
	\includegraphics[width=75mm]{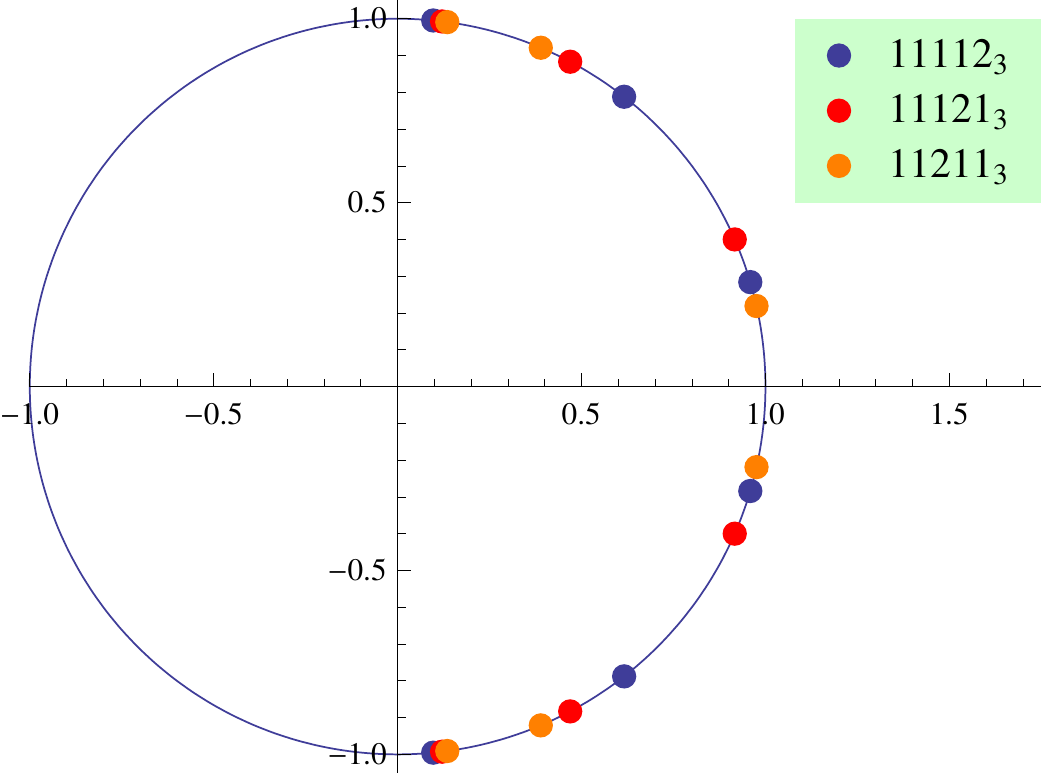} 
	\caption{Roots of the polynomials $Q_{122}$, $Q_{124}$ and~$Q_{130}$.} 
	\label{fRoots}
\end{figure}

Now, if we assume hypotheses \ref{H1}, \ref{hypIrr} and~\ref{hypReal} then 
applying to our polynomials the~transformation 
$$
	z = \kappa_1(z) = i \frac{z-1}{z+1} 
$$
mapping $\Set{R}$ to the unit circle in the complex plane, 
we can define the {\it dual polynomials\/} 
$$
	Q_m(w) = \tP_m(\kappa_1(z)) . 
$$

\begin{hypot}
The polynomials $Q_m(w)$ are self-re\-cip\-ro\-cal polynomials having all roots $\la_j$ on the unit circle 
and in the right-half plane: 
$$
	|\la_j| = 1, \qquad \Re \la_j > 0. 
$$
\end{hypot}

Let us consider, for example, the polynomials 
\begin{gather*}
	\Pz_{122}(z) = z^{6} + 26z^{5} + 120z^{4} + 192z^{3} + 120z^{2} + 26z + 1 
	\\
	\Pz_{124}(z) = z^{6} + 23z^{5} + 119z^{4} + 200z^{3} + 119z^{2} + 23z + 1 
	\\ 
	\Pz_{130}(z) = z^{6} + 22z^{5} + 120z^{4} + 200z^{3} + 120z^{2} + 22z + 1, 
\end{gather*}
corresponding to the configuration 
$$
	122 = \W{11112}_3, \qquad 
	124 = \W{11121}_3, \qquad 
	130 = \W{11211}_3. 
$$
The dual polynomials $Q_{m}(w)$ are % (up to a constant) 
\begin{gather*}
	Q_{122}(w) = { \tfrac{-2i}{486}}\big( \qquad \qquad \\ %
		35w^{6} - 117w^{5} + 209w^{4} - 250w^{3} + 209w^{2} - 117w + 35 \big) 
	\\
	Q_{124}(w) = { \tfrac{-i}{486}}\big( \qquad \qquad \\ %
		77w^{6} - 232w^{5} + 415w^{4} - 496w^{3} + 415w^{2} - 232w + 77 \big) 
	\\ 
	Q_{130}(w) = { \tfrac{-2i}{486}}\big( \qquad \qquad \\ %
		39w^{6} - 117w^{5} + 205w^{4} - 250w^{3} + 205w^{2} - 117w + 39 \big)
	, 
\end{gather*}
and the root of these polynomials are shown on fig.~\ref{fRoots}.

\begin{quest}
What is the asymptotic behaviour of the distributions $\rho_m$? 
\end{quest}

{\it Remark.} In the proof of theorem~\ref{thmExist} we consider, for a given degree~$d$,  
a set of polynomials corresponding to configurations 
$$
	m = \W{1}\; \W{0}^{\ell_1}\; 1 \W{0}^{\ell_2}\; \W{1} \dots \W{0}^{\ell_{d-1}}\; \W{1}_3, 
$$
where ones are separated by long sequences of zeroes. These configurations generate 
sums $\phi^{(m)}$ which are reduced to sums of $d$ almost independent random variables, 
and, in particular, 
$$
	\tP_m(z) \to \frac1{2^{d}}(1 + z)^d, \qquad \ell_j \to \infty. 
$$
Thus, it is easy to see that the corresponding destributions $\rho_m$ converge to the binomial distribution. 

\begin{figure}[ht]
	\includegraphics[width=75mm]{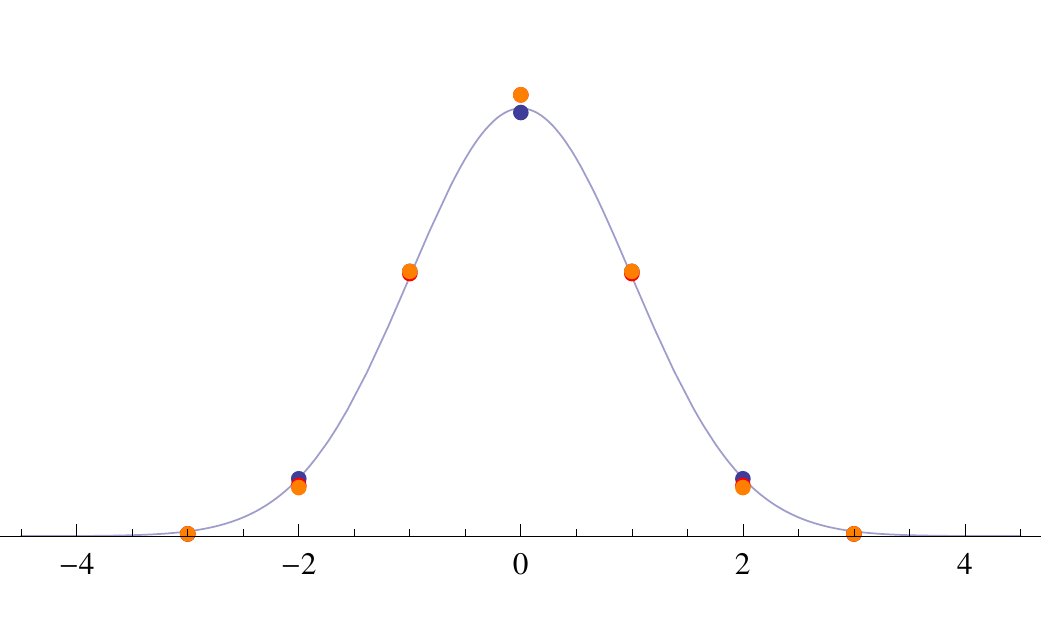} 
	\caption{The distributions $\rho_{122}$, $\rho_{124}$ and~$\rho_{130}$ 
	and the normal distribution.} 
	\label{fDistrRhoM}
\end{figure}

\begin{quest}
Is it true that the distributions $\rho_m$, centered and scaled, 
converge to the normal distribution as ${d(m) \to \infty}$ 
independently on the structure of~$\tP_m()z\:$? (see~fig.~\ref{fDistrRhoM})
\end{quest}

% On fig.~\ref{fDistrRhoM} the distributions $\rho_m$ are shown for ${m = 122,124,130}$. 

%%%%%%%%%%%%%%%%%%%%%%%%%%%%%%%%%%%%%%%%%%%%%%%%%%%%%%%%%%
% Degree
%%%%%%%%%%%%%%%%%%%%%%%%%%%%%%%%%%%%%%%%%%%%%%%%%%%%%%%%%%

\begin{hypot}\label{hypFirst}
The first polynomils $\tP_m(z)$ of degree $d$ is observed at 
\begin{equation}\label{eFirstDegree}
	m = \frac{3^{d-1}+1}{2}. 
\end{equation}
\end{hypot}

\begin{hypot}
Consider a subsequence of polynomials described in hypothesis~\ref{hypFirst}. 
If $m$ is given by formula~\eqref{eFirstDegree} and ${m \in 2\Set{Z}}$ 
then the corresponding polynomials $\Pz_m(z)$ are irreducible monic self-reciprocal polinomials. 
If $m$ is odd then the same is true for $(z+1)^{-1}\Pz_m(z)$. 
The roots $r_j$ of $\Pz_m(z)$ are algebraic integers, moreover, ${r_j \in \Set{R}}$. 
\end{hypot}

\begin{rem}
In particular, if hypothesis~\ref{hypFirst} is true then the following estimate holds 
$$
	d(m) \le 1 + \log_3 (2m-1). 
$$
\end{rem}

\begin{quest}
How $d(m)$ depends on~$m$? 
\end{quest}

% Theorem~\ref{thmIrrCubic} as well as hypotheses {thmIrrCubic} and~

\begin{quest}
Is it true that no one polynomial $\tP_m(z)$ divides another polynomials in this sequence, 
and any $\tP_m(z)$ is never a product of different polynomials $\tP_{m'_k}(z)$ 
of smaller degree? 
\end{quest}

\begin{quest}
Is it true that any operator $\tP_m(\T)$ is not a product of different operators 
${A_j \in \LP}$ in the weak closure of powers of Chacon transformation~$\uT$? 
\end{quest}

\begin{quest}
Can we find $\tP_m(z)$ which is an isolated point in the semigroup generated by all~$\{\tP_{m'}\}$, 
and can we find $\tP_m(\T)$ which is an isolated point in~$\LP$? 
\end{quest}

\begin{quest}
Is it true that the set~$\LP$ does not contain operators given by series 
$$
	\sum_{j \in \Set{Z}} a_j \uT^j, 
$$
where inifinitely many ${a_j \not= 0\:}$? 
Is~it~possible to find among elemets $V \in \LP$ operators of the form 
$$
	V = \varkappa \Theta + \sum_j a_j \uT^j, \qquad \varkappa \not= 0 \: ? 
$$
\end{quest}

The following well-known question still has no answer as well. 

\begin{quest}
Is Chacon$_{(3)}$ transformation is $\varkappa$-mixing, which means that there exists ${V \in \LP}$ 
such that 
$$
	V = \varkappa \Theta + V_2, \qquad \varkappa \not= 0 \: ? 
$$
\end{quest}

\begin{hypot}
There exists ${\eps_0 > 0}$ such that 
among the polynomials $P_m(\uT)$ and in the set~$\LP$ there is no polynomials 
$$
	\sum_{j=0}^d a_j \uT^j, \qquad a_j > 0, 
$$ 
having the property 
$$
	\left| \frac{a_{j+1}}{a_j} - 1 \right| < \eps_0. 
$$ 
\end{hypot}

\begin{quest}
Can we find $\tP_m(z)$ which is an isolated point in the semigroup generated by all~$\{\tP_{m'}\}$, 
and can we find $\tP_m(\T)$ which is an isolated point in~$\LP$? 
\end{quest}

\begin{quest}
How we can describe the entire set $\LP$ for Chacon$_{(3)}$ transformation? 
\end{quest}

%=================================================================================================
\bigskip
\section{Acknowledgements}

The authors are grateful to A.M.\,St\"epin, M.\,Le\-manczyk, J.\,Kwiatkow\-ski, E.\,Janvresse, T.~de~la~Rue, 
El~H.~El~Abdalaoui 
and K.\,Petersen for discussions and helpful remarks around the subject of the paper. 

This work is supported by CNRS (France), 
RFFI grant \No\,11-01-00759-a and 
the grant ``Leading Russian scientific schools'' \No\,NSh-5998.2012.1. 

%\begin{widetext}
%$\quad$
%\end{widetext}

%\newpage
%\section*{References}
%\bigskip

%\begin{center}
%{	\large\bf References}
%\end{center}

%=================================================================================================
\begin{widetext}
%\newpage
\bigskip
\section[Appendix: The limit polynomials]{Appendix: The limit polynomials}

\if0=1
The columns of this table indicate: 
\begin{itemize}
	\item the number $m$
	\item $3$-adic expansion of $m$ (configuration), 
	\item the polynomial $\tP_m(z)$ 
	% \item factorization of $\tP_m(z)$ if it is reducible (up to $-1$) 
\end{itemize} 
\fi

\medskip
\medskip 
\begin{center}
	\bf Table 1. First $122$ limit polynomials $\tP_m(z)$ 
\end{center}
 
The columns of this table indicate: 
the number~$m$, $3$-adic expansion of $m$ (configuration), and the polynomial $\tP_m(z)$. 
We mark by ${}^*$ the idexes corresponding to configurations $\W{111}\dots\W{12}_3$. 
We~skip symmetrical configurations like ${\W{112}_3\sim\W{211}_3}$ 
following Hypothesis~\ref{HSym} which is true in this interval. 

\def\rr{\rule{0pt}{20pt}}
\def\CLine{\\[9pt]}
 
%\label{tPolynomials}
\begin{longtable}{|c|c|c|}
\hline \rr 
% \normalsize 
$\;${\bf Index} $m\;$ & $\;${\bf Configuration}$\;$ & 
	$\;${\bf Polynomial} $\tP_m(z)$ 
	%%& $\;${\bf Factorization}$\;$ & $\;${\bf Zeros}$\;$ 
	\CLine 
\hline\endhead 
\color{blue} 
$1^*$ \rr & \color{blue} $\W{1}_3$ & \color{blue} $\dstyle \tP_1(z) = \tP_3(z) = \tP_9(Z) = \dots = \frac12(1+z)$ %& & 
\CLine 
\hline \rr 
\color{blue} 
$2^*$ & \color{blue} $\W{2}_3$ & \color{blue} $\dstyle \tP_2(z) = \tP_6(z) = \dots = \frac16(1+4z+z^2)$ %& & 
\CLine% 
\hline \rr 
4 & $\W{11}_3$ &  $\dstyle \;\; \tP_{4}(z) = \frac1{9}(2z^{2} + 5z + 2)\;\;$ 
\CLine 
\hline \rr 
\color{blue}5$^*$ & \color{blue}$\W{12}_3$ & \color{blue} $\dstyle \;\; \tP_{5}(z) = \frac1{18}(z^{3} + 8z^{2} + 8z + 1)\;\;$ 
\CLine 
\hline \rr 
8 & $\W{22}_3$ &  $\dstyle \;\; \tP_{8}(z) = \frac1{9}(2z^{2} + 5z + 2)\;\;$ 
\CLine 
\hline \rr 
10 & $\W{101}_3$ &  $\dstyle \;\; \tP_{10}(z) = \frac1{54}(13z^{2} + 28z + 13)\;\;$ 
\CLine 
\hline \rr 
11 & $\W{102}_3$ &  $\dstyle \;\; \tP_{11}(z) = \frac1{54}(4z^{3} + 23z^{2} + 23z + 4)\;\;$ 
\CLine 
\hline \rr 
13 & $\W{111}_3$ &  $\dstyle \;\; \tP_{13}(z) = \frac1{54}(5z^{3} + 22z^{2} + 22z + 5)\;\;$ 
\CLine 
\hline \rr 
\color{blue}14$^*$ & \color{blue}$\W{112}_3$ & \color{blue} $\dstyle \;\; \tP_{14}(z) = \frac1{54}(z^{4} + 13z^{3} + 26z^{2} + 13z + 1)\;\;$ 
\CLine 
\hline \rr 
16 & $\W{121}_3$ &  $\dstyle \;\; \tP_{16}(z) = \frac1{54}(z^{4} + 12z^{3} + 28z^{2} + 12z + 1)\;\;$ 
\CLine 
\hline \rr 
17 & $\W{122}_3$ &  $\dstyle \;\; \tP_{17}(z) = \frac1{54}(4z^{3} + 23z^{2} + 23z + 4)\;\;$ 
\CLine 
\hline \rr 
20 & $\W{202}_3$ &  $\dstyle \;\; \tP_{20}(z) = \frac1{54}(z^{4} + 12z^{3} + 28z^{2} + 12z + 1)\;\;$ 
\CLine 
\hline \rr 
23 & $\W{212}_3$ &  $\dstyle \;\; \tP_{23}(z) = \frac1{54}(5z^{3} + 22z^{2} + 22z + 5)\;\;$ 
\CLine 
\hline \rr 
26 & $\W{222}_3$ &  $\dstyle \;\; \tP_{26}(z) = \frac1{54}(13z^{2} + 28z + 13)\;\;$ 
\CLine 
\hline \rr 
28 & $\W{1001}_3$ &  $\dstyle \;\; \tP_{28}(z) = \frac1{81}(20z^{2} + 41z + 20)\;\;$ 
\CLine 
\hline \rr 
29 & $\W{1002}_3$ &  $\dstyle \;\; \tP_{29}(z) = \frac1{162}(13z^{3} + 68z^{2} + 68z + 13)\;\;$ 
\CLine 
\hline \rr 
31 & $\W{1011}_3$ &  $\dstyle \;\; \tP_{31}(z) = \frac1{162}(17z^{3} + 64z^{2} + 64z + 17)\;\;$ 
\CLine 
\hline \rr 
32 & $\W{1012}_3$ &  $\dstyle \;\; \tP_{32}(z) = \frac1{81}(2z^{4} + 20z^{3} + 37z^{2} + 20z + 2)\;\;$ 
\CLine 
\hline \rr 
34 & $\W{1021}_3$ &  $\dstyle \;\; \tP_{34}(z) = \frac1{162}(4z^{4} + 39z^{3} + 76z^{2} + 39z + 4)\;\;$ 
\CLine 
\hline \rr 
35 & $\W{1022}_3$ &  $\dstyle \;\; \tP_{35}(z) = \frac1{162}(16z^{3} + 65z^{2} + 65z + 16)\;\;$ 
\CLine 
\hline \rr 
38 & $\W{1102}_3$ &  $\dstyle \;\; \tP_{38}(z) = \frac1{162}(5z^{4} + 39z^{3} + 74z^{2} + 39z + 5)\;\;$ 
\CLine 
\hline \rr 
40 & $\W{1111}_3$ &  $\dstyle \;\; \tP_{40}(z) = \frac1{81}(3z^{4} + 20z^{3} + 35z^{2} + 20z + 3)\;\;$ 
\CLine 
\hline \rr 
\color{blue}41$^*$ & \color{blue}$\W{1112}_3$ & \color{blue} $\dstyle \;\; \tP_{41}(z) = \frac1{162}(z^{5} + 19z^{4} + 61z^{3} + 61z^{2} + 19z + 1)\;\;$ 
\CLine 
\hline \rr 
43 & $\W{1121}_3$ &  $\dstyle \;\; \tP_{43}(z) = \frac1{162}(z^{5} + 17z^{4} + 63z^{3} + 63z^{2} + 17z + 1)\;\;$ 
\CLine 
\hline \rr 
44 & $\W{1122}_3$ &  $\dstyle \;\; \tP_{44}(z) = \frac1{81}(2z^{4} + 20z^{3} + 37z^{2} + 20z + 2)\;\;$ 
\CLine 
\hline \rr 
47 & $\W{1202}_3$ &  $\dstyle \;\; \tP_{47}(z) = \frac1{162}(z^{5} + 16z^{4} + 64z^{3} + 64z^{2} + 16z + 1)\;\;$ 
\CLine 
\hline \rr 
50 & $\W{1212}_3$ &  $\dstyle \;\; \tP_{50}(z) = \frac1{162}(5z^{4} + 39z^{3} + 74z^{2} + 39z + 5)\;\;$ 
\CLine 
\hline \rr 
52 & $\W{1221}_3$ &  $\dstyle \;\; \tP_{52}(z) = \frac1{81}(2z^{4} + 18z^{3} + 41z^{2} + 18z + 2)\;\;$ 
\CLine 
\hline \rr 
53 & $\W{1222}_3$ &  $\dstyle \;\; \tP_{53}(z) = \frac1{162}(13z^{3} + 68z^{2} + 68z + 13)\;\;$ 
\CLine 
\hline \rr 
56 & $\W{2002}_3$ &  $\dstyle \;\; \tP_{56}(z) = \frac1{81}(2z^{4} + 18z^{3} + 41z^{2} + 18z + 2)\;\;$ 
\CLine 
\hline \rr 
59 & $\W{2012}_3$ &  $\dstyle \;\; \tP_{59}(z) = \frac1{162}(z^{5} + 17z^{4} + 63z^{3} + 63z^{2} + 17z + 1)\;\;$ 
\CLine 
\hline \rr 
62 & $\W{2022}_3$ &  $\dstyle \;\; \tP_{62}(z) = \frac1{162}(4z^{4} + 39z^{3} + 76z^{2} + 39z + 4)\;\;$ 
\CLine 
\hline \rr 
68 & $\W{2112}_3$ &  $\dstyle \;\; \tP_{68}(z) = \frac1{81}(3z^{4} + 20z^{3} + 35z^{2} + 20z + 3)\;\;$ 
\CLine 
\hline \rr 
71 & $\W{2122}_3$ &  $\dstyle \;\; \tP_{71}(z) = \frac1{162}(17z^{3} + 64z^{2} + 64z + 17)\;\;$ 
\CLine 
\hline \rr 
80 & $\W{2222}_3$ &  $\dstyle \;\; \tP_{80}(z) = \frac1{81}(20z^{2} + 41z + 20)\;\;$ 
\CLine 
\hline \rr 
82 & $\W{10001}_3$ &  $\dstyle \;\; \tP_{82}(z) = \frac1{486}(121z^{2} + 244z + 121)\;\;$ 
\CLine 
\hline \rr 
83 & $\W{10002}_3$ &  $\dstyle \;\; \tP_{83}(z) = \frac1{486}(40z^{3} + 203z^{2} + 203z + 40)\;\;$ 
\CLine 
\hline \rr 
85 & $\W{10011}_3$ &  $\dstyle \;\; \tP_{85}(z) = \frac1{486}(53z^{3} + 190z^{2} + 190z + 53)\;\;$ 
\CLine 
\hline \rr 
86 & $\W{10012}_3$ &  $\dstyle \;\; \tP_{86}(z) = \frac1{486}(13z^{4} + 121z^{3} + 218z^{2} + 121z + 13)\;\;$ 
\CLine 
\hline \rr 
88 & $\W{10021}_3$ &  $\dstyle \;\; \tP_{88}(z) = \frac1{486}(13z^{4} + 120z^{3} + 220z^{2} + 120z + 13)\;\;$ 
\CLine 
\hline \rr 
89 & $\W{10022}_3$ &  $\dstyle \;\; \tP_{89}(z) = \frac1{486}(52z^{3} + 191z^{2} + 191z + 52)\;\;$ 
\CLine 
\hline \rr 
91 & $\W{10101}_3$ &  $\dstyle \;\; \tP_{91}(z) = \frac1{486}(56z^{3} + 187z^{2} + 187z + 56)\;\;$ 
\CLine 
\hline \rr 
92 & $\W{10102}_3$ &  $\dstyle \;\; \tP_{92}(z) = \frac1{486}(17z^{4} + 120z^{3} + 212z^{2} + 120z + 17)\;\;$ 
\CLine 
\hline \rr 
94 & $\W{10111}_3$ &  $\dstyle \;\; \tP_{94}(z) = \frac1{486}(21z^{4} + 121z^{3} + 202z^{2} + 121z + 21)\;\;$ 
\CLine 
\hline \rr 
95 & $\W{10112}_3$ &  $\dstyle \;\; \tP_{95}(z) = \frac1{486}(4z^{5} + 61z^{4} + 178z^{3} + 178z^{2} + 61z + 4)\;\;$ 
\CLine 
\hline \rr 
97 & $\W{10121}_3$ &  $\dstyle \;\; \tP_{97}(z) = \frac1{486}(4z^{5} + 56z^{4} + 183z^{3} + 183z^{2} + 56z + 4)\;\;$ 
\CLine 
\hline \rr 
98 & $\W{10122}_3$ &  $\dstyle \;\; \tP_{98}(z) = \frac1{486}(16z^{4} + 121z^{3} + 212z^{2} + 121z + 16)\;\;$ 
\CLine 
\hline \rr 
100 & $\W{10201}_3$ &  $\dstyle \;\; \tP_{100}(z) = \frac1{243}(8z^{4} + 60z^{3} + 107z^{2} + 60z + 8)\;\;$ 
\CLine 
\hline \rr 
101 & $\W{10202}_3$ &  $\dstyle \;\; \tP_{101}(z) = \frac1{486}(4z^{5} + 55z^{4} + 184z^{3} + 184z^{2} + 55z + 4)\;\;$ 
\CLine 
\hline \rr 
103 & $\W{10211}_3$ &  $\dstyle \;\; \tP_{103}(z) = \frac1{486}(4z^{5} + 59z^{4} + 180z^{3} + 180z^{2} + 59z + 4)\;\;$ 
\CLine 
\hline \rr 
104 & $\W{10212}_3$ &  $\dstyle \;\; \tP_{104}(z) = \frac1{243}(10z^{4} + 60z^{3} + 103z^{2} + 60z + 10)\;\;$ 
\CLine 
\hline \rr 
106 & $\W{10221}_3$ &  $\dstyle \;\; \tP_{106}(z) = \frac1{486}(16z^{4} + 117z^{3} + 220z^{2} + 117z + 16)\;\;$ 
\CLine 
\hline \rr 
107 & $\W{10222}_3$ &  $\dstyle \;\; \tP_{107}(z) = \frac1{486}(52z^{3} + 191z^{2} + 191z + 52)\;\;$ 
\CLine 
\hline \rr 
110 & $\W{11002}_3$ &  $\dstyle \;\; \tP_{110}(z) = \frac1{486}(17z^{4} + 117z^{3} + 218z^{2} + 117z + 17)\;\;$ 
\CLine 
\hline \rr 
112 & $\W{11011}_3$ &  $\dstyle \;\; \tP_{112}(z) = \frac1{243}(11z^{4} + 60z^{3} + 101z^{2} + 60z + 11)\;\;$ 
\CLine 
\hline \rr 
113 & $\W{11012}_3$ &  $\dstyle \;\; \tP_{113}(z) = \frac1{486}(5z^{5} + 61z^{4} + 177z^{3} + 177z^{2} + 61z + 5)\;\;$ 
\CLine 
\hline \rr 
115 & $\W{11021}_3$ &  $\dstyle \;\; \tP_{115}(z) = \frac1{486}(5z^{5} + 59z^{4} + 179z^{3} + 179z^{2} + 59z + 5)\;\;$ 
\CLine 
\hline \rr 
116 & $\W{11022}_3$ &  $\dstyle \;\; \tP_{116}(z) = \frac1{243}(10z^{4} + 60z^{3} + 103z^{2} + 60z + 10)\;\;$ 
\CLine 
\hline \rr 
119 & $\W{11102}_3$ &  $\dstyle \;\; \tP_{119}(z) = \frac1{486}(6z^{5} + 61z^{4} + 176z^{3} + 176z^{2} + 61z + 6)\;\;$ 
\CLine 
\hline \rr 
121 & $\W{11111}_3$ &  $\dstyle \;\; \tP_{121}(z) = \frac1{486}(7z^{5} + 65z^{4} + 171z^{3} + 171z^{2} + 65z + 7)\;\;$ 
\CLine 
\hline \rr 
\color{blue}122$^*$ & \color{blue}$\W{11112}_3$ & \color{blue} $\dstyle \;\; \tP_{122}(z) = \frac1{486}(z^{6} + 26z^{5} + 120z^{4} + 192z^{3} + 120z^{2} + 26z + 1)\;\;$ 
\CLine 
\hline
\end{longtable}

\medskip
\medskip 
\begin{center}
	\bf Table 2. Several remarkable limit polynomials $\tP_m(z)$ for~${m \le 1094}$. 
\end{center}
 
\begin{longtable}{|c|c|c|}
\hline \rr 
% \normalsize 
$\;${\bf Index} $m\;$ & $\;${\bf Configuration}$\;$ & 
	$\;${\bf Polynomial} $\tP_m(z)$ 
	\CLine 
\hline\endhead 
\multicolumn{3}{|c|}{ \large\it First occurence of degree $d$ \rr } 
\CLine 
\hline
%\color{blue} 
$1$ \rr & %\color{blue} 
$\W{1}_3$ & %\color{blue} 
$\dstyle \tP_1(z) = \tP_3(z) = \tP_9(Z) = \dots = \frac12(1+z)$ %& & 
\CLine 
\hline \rr 
%\color{blue} 
$2^*$ & %\color{blue} 
$\W{2}_3$ & %\color{blue} 
$\dstyle \tP_2(z) = \tP_6(z) = \dots = \frac16(1+4z+z^2)$ %& & 
\CLine% 
\hline \rr 
%\color{blue}
5$^*$ & %\color{blue}
$\W{12}_3$ & %\color{blue} 
$\dstyle \;\; \tP_{5}(z) = \frac1{18}(z^{3} + 8z^{2} + 8z + 1)\;\;$ 
\CLine 
\hline \rr 
%\color{blue}
14$^*$ & %\color{blue}
$\W{112}_3$ & %\color{blue} 
$\dstyle \;\; \tP_{14}(z) = \frac1{54}(z^{4} + 13z^{3} + 26z^{2} + 13z + 1)\;\;$ 
\CLine 
\hline \rr 
%\color{blue}
41$^*$ & %\color{blue}
$\W{1112}_3$ & %\color{blue} 
$\dstyle \;\; \tP_{41}(z) = \frac1{162}(z^{5} + 19z^{4} + 61z^{3} + 61z^{2} + 19z + 1)\;\;$ 
\CLine 
\hline \rr 
%\color{blue}
122$^*$ & %\color{blue}
$\W{11112}_3$ & %\color{blue} 
$\dstyle \;\; \tP_{122}(z) = \frac1{486}(z^{6} + 26z^{5} + 120z^{4} + 192z^{3} + 120z^{2} + 26z + 1)\;\;$ 
\CLine 
\hline \rr 
%\color{blue}
365$^*$ & %\color{blue}
$\W{111112}_3$ & %\color{blue} 
$\dstyle \;\; \tP_{365}(z) = \frac1{1458}(z^{7} + 34z^{6} + 211z^{5} + 483z^{4} + 483z^{3} + 211z^{2} + 34z + 1)\;\;$ 
\CLine 
\hline \rr 
%\color{blue}
1094$^*$ & %\color{blue}
$\W{1111112}_3$ & %\color{blue} 
$\dstyle \;\; \tP_{1094}(z) = \frac1{4374}(z^{8} + 43z^{7} + 343z^{6} + 1050z^{5} + 1500z^{4} + 1050z^{3} + 343z^{2} + 43z + 1)\;\;$ 
\CLine 
\hline 
\multicolumn{3}{|c|}{ \large\it Similar configurations \rr } 
\CLine 
\hline \rr 
122 & $\W{11112}_3$ & $\dstyle \;\; \tP_{122}(z) = \frac1{486}(z^{6} + 26z^{5} + 120z^{4} + 192z^{3} + 120z^{2} + 26z + 1)\;\;$ 
\CLine 
\hline \rr 
124 & $\W{11121}_3$ &  $\dstyle \;\; \tP_{124}(z) = \frac1{486}(z^{6} + 23z^{5} + 119z^{4} + 200z^{3} + 119z^{2} + 23z + 1)\;\;$ 
\CLine 
\hline \rr 
130 & $\W{11211}_3$ &  $\dstyle \;\; \tP_{130}(z) = \frac1{486}(z^{6} + 22z^{5} + 120z^{4} + 200z^{3} + 120z^{2} + 22z + 1)\;\;$ 
\CLine 
\hline \rr 
148 & $\W{12111}_3$ &  $\dstyle \;\; \tP_{148}(z) = \frac1{486}(z^{6} + 23z^{5} + 119z^{4} + 200z^{3} + 119z^{2} + 23z + 1)\;\;$ 
\CLine 
\hline \rr 
202 & $\W{21111}_3$ & $\dstyle \;\; \tP_{202}(z) = \frac1{486}(z^{6} + 26z^{5} + 120z^{4} + 192z^{3} + 120z^{2} + 26z + 1)\;\;$ 
\CLine 
\hline 
\multicolumn{3}{|c|}{ \large\it Irreducible up to a root $(-1)$ cubic polynomials \rr } 
\CLine 
\hline \rr 
91 & $\W{10101}_3$ &  $\dstyle \;\; \tP_{91}(z) = \frac1{486}(56z^{3} + 187z^{2} + 187z + 56)\;\;$ 
\CLine 
\hline \rr 
253 & $\W{100101}_3$ &  $\dstyle \;\; \tP_{253}(z) = \frac1{1458}(173z^{3} + 556z^{2} + 556z + 173)\;\;$ 
\CLine 
\hline \rr 
739 & $\W{1000101}_3$ &  $\dstyle \;\; \tP_{739}(z) = \frac1{4374}(524z^{3} + 1663z^{2} + 1663z + 524)\;\;$ 
\CLine 
\hline \rr 
757 & $\W{1001001}_3$ &  $\dstyle \;\; \tP_{757}(z) = \frac1{4374}(533z^{3} + 1654z^{2} + 1654z + 533)\;\;$ 
\CLine 
\hline
\end{longtable}

\newpage 
%\medskip
%\medskip 
\begin{center}
	\bf Table 3. Non-irreducible polynomials $\tP_m(z)$ for~${m \le 122}$ (up to a~root~$-1$). 
\end{center}
 
\begin{longtable}{|c|c|c|}
\hline \rr 
% \normalsize 
$\;${\bf Index} $m\;$ & $\;${\bf Configuration}$\;$ & 
	$\;${\bf Factorization of} $\tP_m(z)$ 
	\CLine 
\hline\endhead 
4 \rr & $\W{11}_3$ & $\dstyle \tP_4(z) = \frac19(2+z)(1+2z)$ 
\CLine 
\hline \rr 
8 & $\W{22}_3$ & $\dstyle \tP_8(z) = \frac19(2+z)(1+2z)$ 
\CLine 
\hline \rr 
28 & $\W{1001}_3$ & $\dstyle \tP_{28}(z) = \frac1{81}(5+4z)(4+5z)$ 
\CLine 
\hline \rr 
40 & $\W{1111}_3$ & $\dstyle \tP_{40}(z) = \frac1{81}(3+5z+z^2)(1+5z+3z^2)$ 
\CLine 
\hline \rr 
52 & $\W{1221}_3$ & $\dstyle \tP_{52}(z) = \frac1{81}(2+6z+z^2)(1+6z+2z^2)$ 
\CLine 
\hline \rr 
56 & $\W{2002}_3$ & $\dstyle \tP_{56}(z) = \frac1{81}(2+6z+z^2)(1+6z+2z^2)$ 
\CLine 
\hline \rr 
68 & $\W{2112}_3$ & $\dstyle \tP_{68}(z) = \frac1{81}(3+5z+z^2)(1+5z+3z^2)$ 
\CLine 
\hline \rr 
80 & $\W{2222}_3$ & $\dstyle \tP_{80}(z) = \frac1{9}(5+4z)(4+5z)$ 
\CLine 
\hline \rr 
244 & $\W{111111}_3$ & $\dstyle \tP_{244}(z) = \frac1{729}(14+13z)(13+14z)$ 
\CLine 
\hline
\end{longtable}

\section{References}

\addvspace{-6pt}

\bibliographystyle{alpha}
\bibliography{AA}

%\bigskip 

\end{widetext}\end{document}